\documentclass[a4paper,11pt,reqno]{amsart}
\usepackage{amsmath, amssymb,amsthm}
\usepackage{enumerate}
\usepackage{cite}
\usepackage{units}
\numberwithin{equation}{section}
\newtheorem{theorem}{Theorem}[section]
\newtheorem{lemma}{Lemma}[section]

\newtheorem{corollary}{Corollary}[section]
\newtheorem{remark}{Remark}[section]
\newtheorem{definition}{Definition}[section]

\theoremstyle{definition}
\newtheorem{example}{Example}

\begin{document}
\bibliographystyle{amsplain}
\title{{{
Extremal mild solutions for Hilfer fractional evolution equation with mixed monotone Impulsive conditions}}}
\author{Divya Raghavan*
}
\address{
Department of Mathematics,
Indian Institute of Technology Roorkee, Roorkee-247667,
Uttarakhand, India
}
\email{divyar@ma.iitr.ac.in,madhanginathan@gmail.com}
\author{
N. Sukavanam
}
\address{
Department of  Mathematics  \\
Indian Institute of Technology Roorkee,  Roorkee-247667,
Uttarakhand,  India
}
\email{n.sukavanam@ma.iitr.ac.in}
\bigskip
\begin{abstract}
The well established mixed monotone iterative technique that is used to study the existence and uniqueness of fractional order system is studied explicitly for impulsive system with Hilfer fractional order in this paper. The procedure of finding mild $L$-quasi solution of such impulsive evolution equation with noncomapct semigroups involves measure of non-compactness and Sadovskii's fixed point theorem as well. An example is provided to illustrate the main results.
\end{abstract}
\subjclass[2010]{26A33; 34K30; 34K45; 47D06}
\keywords{Lower and upper solution; Impulsive system; Hilfer fractional derivative; non-compact measure}
\maketitle
\pagestyle{myheadings}
\markboth
{Divya Raghavan and N. Sukavanam}
{Extremal mild solutions for Hilfer fractional evolution equation with mixed monotone Impulsive conditions}
\section{Literature Motivation}
\label{Intro}
Over the years, the basic study of differential equation mainly leads to the finding of its extremal solution. Conditions on when a differential system will have a unique solution is always a challenge for the researchers. Various methods are used to investigate the uniqueness of solution for the desired differential system. Successive approximation or the iterative technique are the typical methods used to determine such unique solution. Du and Lakshmikanthan\cite{Monotone-first-paper} constructed an iterative procedure for an initial value problem,
\begin{align*}
x^{'}=g(t,x); \enspace x(0)=x_{0},
\end{align*}
where the unique solution was guaranteed by upper and lower solution in a closed set. Thereafter, many articles emerged in this direction based on their paper. The case when the order is preserved with respect to the image of the function, then it is monotonicity property. Whereas, the case leading to the decomposition of the function into monotonically decreasing function and monotonically increasing function guided Guo and Lakshimkantham \cite{Mixed-first} to introduce the concept of coupled fixed point. Their work focussed on the existence criteria for both continuous and discontinuous operators defined as $A:D\times D\rightarrow E$ where $D$ is the subset of the Banach space $E$ which is partially ordered by a cone $N$ (details regarding cone can be referred to \cite{Mixed-Cone} by Guo and Lakshmikantham) and $A(x,y)$ is non-decreasing in $x$ and non-increasing in $y$ which were termed as mixed monotone. Mixed monotonicity property finds itself useful mainly in the convergence analysis, global stability analysis, qualitative analysis etc. In continuation, Guo \cite{Mixed-fixed} investigated the existence and uniqueness of mixed monotone operator for a general case where the operator need not be continuous. Chang and Ma \cite{Mixed-Chang1} studied the existence of coupled fixed points of set valued operators defined as $A:D\times D\rightarrow 2^{E}$. As an application, the authors considered the functional equation
\begin{align}
\label{eqn:functional}
g(x)=\sup_{y\in D}[f(x,y)+F(x,y,g(T(x,y)))],\enspace x\in S,
\end{align}
with $S$ the state space, $D$ the decision space, $R=(-\infty,\infty)$, $T:S\times D\rightarrow S$, $f:S\times D\rightarrow R$ and
$F:S\times D\times R\rightarrow R$ emerged in dynamic programming. Similarly Sun and Liu \cite{Mixed-Iterative} improved the existing results on conditions on operator, where their conditions don't require the operator to be continuous as well as the cone $N$ to be normal. The authors implemented their conditions to nonlinear Hammerstein integral equation given by
\begin{align*}
g(x)=\int_{G}k(x,y)f(y,g(y))dy,
\end{align*}
with $G$ the bounded closed subset of $\mathbb{R}^{n}$, $k(x,y):G\times G \rightarrow \mathbb{R}^{1}$ an non negative operator and $f(y,g(y))=f_{1}(y,g(y))+f_{2}(y,g(y))$ with $f_{1}(y,g(y))$ is non-decreasing in $x$ and $f_{2}(y,g(y))$ is non-increasing in $y$. Mean while Chang and Guo \cite{Mixed-Chang2} studied the existence and uniqueness for multiple mixed monotone operators. These multiple mixed monotone operators finds application for example, in system of functional equation in dynamic programming of multistage decision process given as
\begin{align*}
\left\{
  \begin{array}{ll}
  \displaystyle f(x)=\sup_{y\in D}[\phi(x,y)+G(x,y,g(T(x,y)))],\\
   \displaystyle g(x)=\sup_{y\in D}[\phi(x,y)+F(x,y,f(T(x,y)))],
      \end{array}
\right.
\end{align*}
where the operators are defined as in \eqref{eqn:functional}. Zhang \cite{Mixed-Concave1} analysed the case when the operator is convex in nature. To be more precise, Zhang considered the operator, say $A$, satisfy
$$A(tx+(1-t)y)\leq tAx+(1-t)Ay,
$$
 for $x,y\in D(A)$ with $x\leq y$ and $t\in [0,1]$. It is to be noted that $A$ is concave if $-A$ is convex. The author gave the existence and uniqueness of fixed points of such convex-concave mixed monotone operators. Application to differential equations in Banach spaces and to nonlinear integral equations in an unbounded region are given by Zhang to assert the theory developed. A general method of finding the fixed points of such $\phi$ concave-$(-\psi)$ convex operators was not developed until Xu and Jia \cite{Mixed-concave-convex} gave a typical approach for such mixed monotone operators and gave a definition for the operator $A$ to be $\phi$ concave-($-\psi$)convex. As an application, the authors studied a nonlinear integral equation in bounded/unbounded region. Soon the mixed monotone operators found themselves effective in impulsive system as well. Chen \cite{Mixed-imp-BS} employed mixed monotone iterative technique to deal with the existence of solution of the impulsive periodic boundary value problem in a Banach space $E$ given below with the help of coupled upper and lower $L$-quasi solutions.
\begin{align*}
\left\{
  \begin{array}{ll}
x^{'}(t)=g(t,x(t),x(t)), \enspace t\in [0,\omega], \enspace t\neq t_{k},\\
\Delta x|_{t=t_{k}}=J_{k}(x(t_{k}),x(t_{k})), \enspace k=1,2,\ldots,l,\\
x(0)=x(\omega).
      \end{array}
\right.
\end{align*}

 Here, $g\in C([0,\omega]\times E\times E,E)$; $0<t_{1}<t_{2}<\ldots<t_{l}<\omega$; $J_{k}\in C(E\times E,E)$ an impulsive function; $\Delta x|_{t=t_{k}}$ is the jump at $t=t_{k}$. Chen and Li\cite{Mixed-semi-BS} studied the existence of mild $L$-quasi solutions for IVP. Their results enhanced some related theory in ordinary differential equation and partial differential equation and generalised many previous results in this direction.

Fractional order differential equation are more practical oriented, which can be realised with the increase in the number of the research articles with real time applications. It is natural to study the existence of solution using mixed monotone operators for fractional order system. Chen and Li \cite{Mixed-frac-nonlocal} considered a fractional nonlocal evolution system of the form
\begin{align}
\label{eqn:NPFEE}
\left\{
  \begin{array}{ll}
{}^{C}D^{q}_{t}x(t)+Ax(t)=g(t,x(t),x(t)), \enspace t\in J=[0,\omega],\\
x(0)=f(x,x),
      \end{array}
\right.
\end{align}
where ${}^{C}D^{q}_{t}$ is the Caputo derivative of order $q\in (0,1)$; $A:D(A)\subset E \rightarrow E$ a closed linear operator such that $-A$ generates $Q(t)$ a uniformly bounded $C_{0}$-semigroup; $g\in C(J\times E\times E,E)$ and $f(x,x)$ is a nonlocal function. Using operator semigroup theory, theory of mixed monotone operator and some perturbation method, the authors obtained the coupled minimal and maximal mild $L$-quasi solutions of \eqref{eqn:NPFEE}. It can be noted that, when $L\equiv 0$, then the coupled minimal and maximal mild $L$-quasi solutions are equivalent to minimal and maximal mild quasi solutions of \eqref{eqn:NPFEE}. For semilinear evolution equation with impulsive conditions, Li and Gou \cite{Mixed-Semi-Imp} studied the existence and mild $L$-quasi solutions of the periodic boundary value problem given by
\begin{align*}
\left\{
  \begin{array}{ll}
{}^{C}D^{q}_{t}x(t)=g(t,x(t),x(t)), \enspace t\in [0,\omega], \enspace t\neq t_{k},\\
\Delta x|_{t=t_{k}}=J_{k}(x(t_{k}),x(t_{k})), \enspace k=1,2,\ldots,l,\\
x(0)=x(\omega),
      \end{array}
\right.
\end{align*}
where the operator, derivative and the function $g$ are defined as in \eqref{eqn:NPFEE}; $J_{k}$ an impulsive function; $0=t_{0}<t_{1}<\ldots<t_{l}<\omega$; $\Delta x|_{t=t_{k}}= x(t_{k}^{+})-x(t_{k}^{-})$, where $x(t_{k}^{+})$ and $x(t_{k}^{-})$ corresponds to right and left limit of $x(t)$ at $t=t_{k}$. Another notable work is the work of Zhao and Wang \cite{Mixed-impulsive-Zhang} on the mixed monotone technique for fractional impulsive system of Caputo order.

The existence of solution of Hilfer fractional derivative which acts as a interim between the two classical fractional derivative Caputo and Riemann-Liouville derivative which was outlined by Hilfer \cite{Hilfer-book} was examined by Furati et al. \cite{Hilfer-exist-1}. Consequently, Guo and Li \cite{Monotone-frac-Hilfer} utilized fixed point theorem together with monotone iterative technique to study the existence of extremal solution of Hilfer fractional nonlocal evolution equation. Recently, Guo et al. investigated the existence of mild $L$-quasi-solutions using fixed point theorem along with mixed monotone iterative technique of nonlocal Hilfer fractional evolution equation defined as
\begin{align*}
\left\{
  \begin{array}{ll}
D^{p,q}_{0+}x(t)+Ax(t)=g(t,x(t),G(x(t))), \enspace t\in (0,\omega],\\
I_{0+}^{(1-p)(1-q)}x(0)=x_{0}+\sum_{k=1}^{l}\lambda_{k}x(\tau_{k}), \enspace \tau_{k}\in (0,\omega],
\end{array}
\right.
\end{align*}
where $D^{p,q}_{0+}x(t)$ denotes the Hilfer fractional derivative of order $q$ and type $p$ with $\frac{1}{2}<q<1$ and $0\leq p\leq 1$; the operator $G$ is defined as $Gx(t)=\int_{0}^{t}K(t,s)x(s)ds$; $\tau_{k}$ are prefixed points with $k=1,2,\ldots,l$ satisfying $0\leq \tau_{1}\leq\tau_{2}\leq \ldots \leq\tau_{l}<\omega$; $\lambda_{k}$ are real numbers. But many physical problems are impulsive in nature, it is ideal to study an impulsive system. It is challenging to work on the existence of solution of impulsive fractional system with Hilfer fractional order. One can refer for the articles on approximate controllability of impulsive system with Hilfer fractional derivative by Debbouche and Antonov \cite{Hilfer-impulsive-inclusion}, by  Ahmed et al. \cite{Hilfer-impulsive-Ahmed}and by Du et al. \cite{Hilfer-impulsive-Du}. Hence it is worth to study the existence of coupled mild $L$-quasi solution of Hilfer fractional impulsive system as below
\begin{align}
\label{eqn:Mixed-impul-Hilfer}
\left\{
  \begin{array}{ll}
   D_{0+}^{\mu,\nu}x(t)+Ax(t)= g(t,x(t),x(t)), \enspace t \in J=[0,T], \enspace t\neq t_{k},\\
   \Delta I_{t_{k}}^{(1-\lambda)}x(t_{k})=\phi_{k}(x(t_{k}),x(t_{k})),\enspace k=1,2,\ldots l,\\
    I_{0+}^{(1-\lambda)}[x(0)]= x_{0},
      \end{array}
\right.
\end{align}
\noindent where $D_{0+}^{\mu,\nu}$ denotes the Hilfer fractional derivative of order $0<\mu<1$ of type $0\leq \nu \leq1$ and
$\lambda=\mu+\nu-\mu \nu$; $A:D(A)\subseteq E\rightarrow E$  is a closed linear operator and $-A$ generates a $C_{0}$-semigroup $Q(t)(t\geq 0)$ on a
Banach space $E$. Let the impulse effect takes place at $t=t_{k}$, for $(k=1,2,\ldots,l)$; $\phi_{k}\in C(E\times E,E)$ determines the size of the jump at time $t_{k}$. In other words, the impulsive moments meet the relation $\Delta I_{t_{k}}^{1-\lambda}x(t_{k})=I_{t_{k}^{+}}^{1-\lambda}x(t_{k}^{+})-I_{t_{k}^{-}}^{1-\lambda}x(t_{k}^{-})$,
where $I_{t_{k}^{+}}^{1-\lambda}x(t_{k}^{+})$ and $I_{t_{k}^{-}}^{1-\lambda}x(t_{k}^{-})$ denotes the right and the left limit of
$I_{t_{k}}^{1-\lambda}x(t)$ at $t=t_{k}$ with  $0=t_{0}<t_{1}\ldots<t_{l}<t_{l+1}=T$; $g\in C(J\times E \times E,E)$; $x_{0}\in E$.

The rest of the paper is organised in the following way. Section 2 includes essential definitions and lemma for the main results, while Section 3 encloses the main result under appropriate assumptions. Finally Section 4 possess an example to ascertain the main results.
\section{Essential notions}
\label{Preliminaries}
This section includes some basic results, definitions and lemmas that are relevant to this paper.
\begin{definition}{\rm{\cite{Mixed-Cone}}}
Let $E$ be a real Banach space. A nonempty convex closed set $N\subset E$ is called a cone if it satisfies the following two conditions:
\begin{enumerate}
\item
$x\in N$, $\eta\geq \theta$ $\Rightarrow$ $\eta x\in N$.
\item
$x\in N$, $-x \in N$ $\Rightarrow$ $x=\theta$, where $\theta$ denotes the zero element of $E$.
\end{enumerate}
Every cone $N$ in $E$ defines a partial ordering in $E$ given by $x\leq y$ $\iff$ $y-x\in N$.
\end{definition}
\begin{definition}{\rm{\cite{Mixed-Cone}}}
A cone $N$ is said to be normal if there exists a positive constant $\tilde{N}$ such that $\forall$, $x,y\in N$,
\begin{align*}
\theta\leq x\leq y \enspace \Rightarrow \|x\|\leq \tilde{N}\|y\|.
\end{align*}
\end{definition}
It is to be noted that a normal cone is always convex and a cone is said to be positive if for $x,y\in N$, $y-x\in N$ $\forall$ $x<y$.
\begin{definition}{\rm{\cite{Mixed-fixed}}}
An operator $A:D\times D\rightarrow E$ with $D\subset E$ is said to be mixed monotone if $A(x,y)$ is non-decreasing in $x$ and non-increasing in $y$, that is for $x_{1},x_{2},y \in D$, if $x_{1}\leq x_{2}$ then, $A(x_{1},y)\leq A(x_{2},y)$. Similarly, for $y_{1},y_{2},x \in D$, if $y_{1}\leq y_{2}$ then, $A(x,y_{1})\geq A(x,y_{2})$. Also a point $(\hat{x},\hat{y}) \in D\times D$ is called a coupled fixed point of $A$ if $A(\hat{x},\hat{y})=\hat{x}$ and $A(\hat{y},\hat{x})=\hat{y}$ and $\hat{x}$ is the fixed point of $A$ if $A(\hat{x},\hat{x})=\hat{x}$.
\end{definition}

Let $C(J,E)$ denote the space of all $E$-valued continuous function from $J$ to $E$ which is an ordered Banach space generated by the positive cone $N=\{x\in E|x(t)\geq \theta\}$. Also let $C_{1-\lambda}(J,E)$ be defined as $C_{1-\lambda}(J,E)=\{x:t^{1-\lambda}x(t)\in C(J,E)\}$. Clearly $C_{1-\lambda}(J,E)$ is an ordered Banach space induced by the positive cone $N^{'}=\{x\in C_{1-\lambda}(J,E)|x(t)\geq \theta, t\in J\}$. Here $N$ and $N^{'}$ both are normal with the same normal constant $\tilde{N}$. Let $PC(J,E)$ be an ordered Banach space defined as $PC(J,E)=\{x:J\rightarrow E, x(t)\enspace \mbox{is continous at}\enspace t\neq t_{k} \enspace \mbox{and} \enspace x(t_{k}^{+}) \enspace \mbox{exists},\enspace k=1,2,\ldots,l\}$, with the norm $\|x\|_{PC}=\sup\{\|x(t)\|:t\in J\}$. As an impulsive system is considered, a piecewise continuous Banach space should be defined. Let $PC_{1-\lambda}(J,E)=\{x:(t-t_{k})^{1-\lambda}x(t) \in C((t_{k},t_{k+1}],E)$ and $\displaystyle \lim_{t\rightarrow t_{k+}}(t-t_{k})^{1-\lambda}x(t)$, $k=1,2,\ldots l$ exists with the norm
\begin{align*}
\|x(t)\|_{PC_{1-\lambda}}=\max\{\sup_{t\in(t_{k},t_{k+1}]}(t-t_{k})^{1-\lambda}\|x(t)\|:\enspace k=0,1,\ldots,l\}.
\end{align*}

The fractional integral of order $\mu$ for an integrable function $g$ is given as \cite{Podlubny-book},
\begin{align*}
I^{\mu}_{t}g(t)=\frac{1}{\Gamma(\mu)}\int^{t}_{0}(t-s)^{\mu-1}g(s)ds, \enspace \enspace 0< \mu <1.
\end{align*}
 Here $\Gamma(\cdot)$ is the gamma function. Also, the fractional derivative of Caputo and
 Riemann-Liouville of order $\mu$, respectively are given by \cite{Podlubny-book},
\begin{align*}
^{C}D^{\mu}_{0+}g(t)=\frac{1}{\Gamma(1-\mu)}\int^{t}_{0}\frac{g'(s)}{(t-s)^{\mu}}ds, \enspace t>0,\enspace 0< \mu <1,
\end{align*}
and
\begin{align*}
^{L}D^{\mu}_{0+}g(t)=\frac{1}{\Gamma(1-\mu)}\left(\frac{d}{dt}\right)\int^{t}_{0}\frac{g(s)}{(t-s)^{\mu}}ds,\enspace t>0,\enspace 0< \mu <1.
\end{align*}
The Hilfer fractional derivative of order $0< \mu <1$ and type $0\leq \nu \leq 1$ of function $g(t)$ is given as
\begin{align*}
D^{\mu,\nu}_{0+}g(t)=I_{0+}^{\nu(1-\mu)}DI_{0+}^{(1-\nu)(1-\mu)}g(t),
\end{align*}
where $D:=\frac{d}{dt}$.  Gu and Trujillo \cite{Hilfer-remark} can be referred for more details on Hilfer fractional derivative. Moreover, Riemann-Liouville and Caputo can be regarded as a special case of Hilfer fractional derivative, respectively as
\begin{align*}
D_{0+}^{\mu,\nu}=
\left\{
  \begin{array}{ll}
   DI_{0+}^{1-\mu}={ }^{L}D_{0+}^{\mu},\enspace \nu=0\\
   I_{0+}^{1-\mu}D={ }^{C}D^{\mu}_{0+},\enspace \nu=1.
  \end{array}
\right.
\end{align*}
The parameter $\lambda$ satisfies $\lambda=\mu+\nu-\mu \nu, \enspace 0<\lambda\leq 1$.
\begin{definition}{\rm{\cite{Mixed-impulsive-Zhang}}}
An operator family $Q(t):E\rightarrow E$ for $t\geq 0$ is supposedly positive if, for any $u\geq N$, the inequality $Q(t)u\geq \theta$ holds.
\end{definition}

It can be referred \cite{Monotone-non-compact} that the Kuratowski measure of non-compactness measure denoted by $\alpha(\cdot)$ is defined on a bounded set. For any $t\in J$ and $B \subset C(J,E)$, define $B (t)=\{x(t):x\in B\}$. If $B$ is bounded in $C(J,E)$, then $B$ is bounded in $E$. Also, $\alpha(B(t))\leq \alpha(B)$.

The following Lemmas are necessary for the proof of the main theorem in the next section.
\begin{lemma}{\rm{\cite{Mixed-Hilfer-Nonlocal}}}
\label{lem:Mixed lem-1}
Let $E$ be a Banach space and let $D\subset E$ be bounded. Then there exists a countable set $D_{0}\subset D$ such that $\alpha(D)\leq 2\alpha(D_{0})$.
\end{lemma}
\begin{lemma}{\rm{\cite{Monotone-non-compact}}}
\label{lem:Mixed-cgt}
Let $B_{p}=\{x_{p}\}\subset C(J,E),\enspace (p=1,2,\ldots)$ be a bounded and countable set. Then $\alpha(B_{p}(t))$ is Lebesgue integral on $J$, and
$$
\alpha\left(\Big\{\int_{J}x_{p}(t)dt|_{p=1,2,\ldots,}\Big\}\right)\leq 2\int_{J}\alpha(B_{p}(t))dt.
$$
\end{lemma}
\begin{lemma}{\rm{\cite{Mixed-Hilfer-Nonlocal}}}
\label{lem:Mixed-lem-3}
Let $E$ be a Banach space and let $D \subset C([b_{1},b_{2}],E)$ be bounded and equicontinuous. Then $\alpha(D(t))$ is continuous on $[b_{1},b_{2}]$ and
\begin{align*}
\displaystyle \alpha(D)=\max_{t\in [b_{1},b_{2}]}\alpha(D(t)).
\end{align*}
\end{lemma}
\begin{lemma}{\rm{\cite{Sadovskii theorem}}} {\rm{(Sadovskii fixed point theorem)}}
\label{lem:Sadovskii}
Let $E$ be a Banach space and $\Omega$ be a nonempty bounded convex closed set in $E$. If $\mathcal{Q}:\Omega\rightarrow \Omega$ is a condensing mapping, then $\mathcal{Q}$ has a fixed point in $\Omega$.
\end{lemma}
The subsequent Lemma is with reference to the generalized Gronwall inequality for fractional differential equation.
\begin{lemma}{\rm{\cite{Gronwall-Inequality}}}
\label{lem:Gronwall}
Suppose $b\geq 0$, $\beta> 0$ and $a(t)$ is a nonnegative function locally integrable on $0\leq t <T$ (some $T\leq +\infty$), and suppose $x(t)$ is nonnegative and locally integrable on $0\leq t <T$ with
\begin{align*}
x(t)\leq a(t)+b\int_{0}^{t}(t-s)^{\beta-1}x(s)ds
\end{align*}
on this interval; then
\begin{align*}
x(t)\leq a(t)+\int_{0}^{t}\Big[\sum_{n=1}^{\infty}\frac{(b\Gamma(\beta))^{n}}{\Gamma(n\beta)}(t-s)^{n\beta-1}a(s)\Big]ds, \enspace 0\leq t <T.
\end{align*}
\end{lemma}
\begin{definition}{\rm{\cite{Hilfer-impulsive-inclusion}}}
\label{Hilfer-mild-solution}
A function $x \in PC_{1-\lambda}(J,E)$ is called the mild solution of system \eqref{eqn:Mixed-impul-Hilfer}, if for $ t \in J$
it satisfies the following integral equation
\begin{align}
\label{eqn:Mixed-solution-origi}
x(t)=S_{\mu,\nu}(t)x_{0}+&\displaystyle \sum_{i=1}^{k}S_{\mu,\nu}(t-t_{i})\phi_{i}(x(t_{i}),x(t_{i}))\nonumber\\
&+\int_{0}^{t}(t-s)^{\mu-1}P_{\mu}(t-s)g(s,x(s),x(s))ds
\end{align}
where,
\begin{align*}
S_{\mu,\nu}(t)= I_{0+}^{\nu(1-\mu)}K_{\mu}(t),\enspace K_{\mu}(t)=t^{\mu-1}P_{\mu}(t),\enspace P_{\mu}(t)=
\int_{0}^{\infty}\mu \theta \xi_{\mu}(\theta)Q(t^{\mu}\theta)d\theta,
\end{align*}
\begin{align*}
\varpi_{\mu}(\theta)= \frac{1}{\pi} \sum_{n=1}^{\infty}(-1)^{n-1}\theta^{-n\mu-1}
\frac{\Gamma(n\mu+1)}{n!}\sin(n\pi\mu),\enspace \theta\in (0,\infty)
\end{align*}
and $\xi_{\mu}(\theta)= \frac{1}{\mu}\theta^{-1-\frac{1}{\mu}}\varpi_{\mu}(\theta^{-\frac{1}{\mu}})$ is a probability density function defined on $(0,\infty)$, that is
\begin{align*}
\xi_{\mu}(\theta)\geq 0 \enspace \mbox{and} \int^{\infty}_{0}\xi_{\mu}(\theta)d\theta=1.
\end{align*}
\end{definition}
\begin{remark}{${}$}
\begin{enumerate}
\item
From{\rm\cite{Hilfer-exist-1}}, when $\nu=0$, the solution reduces to the solution of classical
Riemann-Liouville fractional derivative, that is, $S_{\mu,0}(t)=K_{\mu}(t)$.
\item
Similarly when $\nu=1$, the solution reduces to the solution of classical
Caputo fractional derivative, that is $S_{\mu,1(t)}=S_{\mu}(t)$.
\end{enumerate}
\end{remark}
\begin{lemma}{\rm{\cite{Hilfer-impulsive-inclusion}}}
\label{lem:Relax-bounds}
If the analytic semigroup $Q(t)(t\geq 0)$ is bounded uniformly, then the operator $P_{\mu}(t)$ and $S_{\mu,\nu}(t)$ satisfies the following bounded and
continuity conditions.
\begin{enumerate}
\item
$S_{\mu,\nu}(t)$ and $P_{\mu}(t)$ are linear bounded operators and for any $x\in E$,
\begin{align*}
\|S_{\mu,\nu}(t)x\|_{E}\leq \frac{M t^{\lambda-1}}{\Gamma(\lambda)}\|x\|_{E}\enspace \mbox{and}\enspace \|P_{\mu}(t)x\|_{E}\leq \frac{M}{\Gamma(\mu)}\|x\|_{E}.
\end{align*}
\item
$S_{\mu,\nu}(t)$ and $P_{\mu}(t)$ are strongly continuous, which means that for any $x\in E$ and $0<t^{'}<t^{''}\leq T$,
\begin{align*}
\|P_{\mu}(t')x-P_{\mu}(t'')x\|_{E}\rightarrow 0 \enspace \mbox{and}\enspace
\|S_{\mu,\nu}(t')x-S_{\mu,\nu}(t'')x\|_{E}\rightarrow 0 \enspace \mbox{as}\enspace t''\rightarrow t'.
\end{align*}
\end{enumerate}
\end{lemma}
\section{Main Results}
\label{Main results}
To prove the main results of this paper, a perturbed equivalent system with constant $C\geq 0$ given below is taken for consideration.
\begin{align}
\label{eqn:Mixed-alternate}
\left\{
  \begin{array}{ll}
   D_{0+}^{\mu,\nu}x(t)+(A+CI)x(t)= g(t,x(t),x(t))+Cx(t), \enspace t \in J, \enspace t\neq t_{k},\\
    \Delta I_{t_{k}}^{(1-\lambda)}x(t_{k})=\phi_{k}(x(t_{k}),x(t_{k})),\enspace k=1,2,\ldots l,\\
    I_{0+}^{(1-\lambda)}[x(0)]= x_{0}.
      \end{array}
\right.
\end{align}
\begin{remark}{${}$}
\label{rem:Mixed-Remark-alternate}
\begin{enumerate}
\item
With reference to {\rm{\cite{Monotone-semigroup}}}, for any $C\geq 0$, $-(A+CI)$ generates an analytic semigroup $R(t)=e^{-Ct}Q(t)$ and for $t\geq 0$, $R(t)$ is positive and $\displaystyle\sup_{t\in[0,\infty)}\|R(t)\|\leq M^{*}$ for $M^{*}\geq 1$.
\item
Let $S^{*}_{\mu,\nu}(t)$ and $P^{*}_{\mu}(t)$ for $t\geq 0$ be two families of operators defined by
\begin{align*}
S^{*}_{\mu,\nu}(t)=&I_{0+}^{\nu(1-\mu)}K^{*}_{\mu}(t),\quad K^{*}_{\mu}(t)=t^{\mu-1}P^{*}_{\mu}(t),\\
&P^{*}_{\mu}(t)=\int_{0}^{\infty}\mu \theta \xi_{\mu}(\theta)R(t^{\mu}\theta)d\theta.
\end{align*}
\item
The above two operators are positive for $(t\geq 0)$ and for any $x\in E$,
\begin{align*}
\|S^{*}_{\mu,\nu}(t)\|\leq \dfrac{M^{*} t^{\lambda-1}}{\Gamma(\lambda)}\enspace\|P^{*}_{\mu}(t)\|\leq \dfrac{M^{*}}{\Gamma(\mu)}\enspace \mbox{and}\enspace
\|K^{*}_{\mu}(t)\|\leq \dfrac{M^{*}t^{\mu-1}}{\Gamma(\mu)}.
\end{align*}
\item
$S^{*}_{\mu,\nu}(t)$ and $P^{*}_{\mu}(t)$ are strongly continuous, which means that for any $x\in E$, $0<t^{'}<t^{''}\leq T$, and as $t''\rightarrow t'$,
\begin{align*}
\|P^{*}_{\mu}(t')x-P^{*}_{\mu}(t'')x\|_{E}\rightarrow 0 \enspace \mbox{and}\enspace
\|S^{*}_{\mu,\nu}(t')x-S^{*}_{\mu,\nu}(t'')x\|_{E}\rightarrow 0.
\end{align*}
\end{enumerate}
\end{remark}
\begin{definition}
A function $x\in PC_{1-\lambda}(J,E)$ is said to be a mild solution of the problem {\rm{\eqref{eqn:Mixed-alternate}}} if $x$ satisfies the following integral equation.
\begin{align*}
x(t)=&S^{*}_{\mu,\nu}(t)x_{0}+\displaystyle \sum_{i=1}^{k}S^{*}_{\mu,\nu}(t-t_{i})\phi_{i}(x(t_{i}),x(t_{i}))\\
&+\int_{0}^{t}(t-s)^{\mu-1}P^{*}_{\mu}(t-s)\Big[g(s,x(s),x(s))+Cx(s)\Big]ds.
\end{align*}
\end{definition}
For $y,z\in PC_{1-\lambda}(J,E)$, $[y,z]$ is used to denote the order interval $\{x\in PC_{1-\lambda}(J,E):y\leq x\leq z\}$ and for $t\in J$, $[y(t),z(t)]$ denotes the order interval $\{x(t)\in PC_{1-\lambda}(J,E):y(t)\leq x(t)\leq z(t)\}$.
\begin{definition}
\label{eqn:Mixed-impul-Hilfer-upper}
If $y_{0}, z_{0}\in PC_{1-\lambda}(J,E)$ satisfies all the inequalities of
\begin{align*}
\left\{
  \begin{array}{ll}
   D_{0+}^{\mu,\nu}y_{0}(t)+Ay_{0}(t)\leq g(t,y_{0}(t),z_{0}(t))+L(y_{0}(t)-z_{0}(t)), \enspace t \in J, \enspace t\neq t_{k}\\
     \Delta I_{t_{k}}^{(1-\lambda)}y_{0}(t_{k})\leq\phi_{k}(y_{0}(t_{k}),z_{0}(t_{k})),\enspace k=1,2,\ldots l\\
    I_{0+}^{(1-\lambda)}[y_{0}]\leq x_{0}.
      \end{array}
\right.
\end{align*}
\begin{align*}
\left\{
  \begin{array}{ll}
   D_{0+}^{\mu,\nu}z_{0}(t)+Az_{0}(t)\geq g(t,z_{0}(t),y_{0}(t))+L(z_{0}(t)-y_{0}(t)), \enspace t \in J, \enspace t\neq t_{k}\\
     \Delta I_{t_{k}}^{(1-\lambda)}z_{0}(t_{k})\geq\phi_{k}(y_{0}(t_{k}),z_{0}(t_{k})),\enspace k=1,2,\ldots l\\
    I_{0+}^{(1-\lambda)}[z_{0}]\geq x_{0}.
      \end{array}
\right.
\end{align*}
 for a constant $L\geq 0$, then $y_{0}$ and $z_{0}$ are called the coupled lower and upper $L$-quasi solution of the problem {\rm{\eqref{eqn:Mixed-impul-Hilfer}}}. If the inequalities are replaced by equality, then $y_{0},z_{0}$ are called coupled $L$-quasi solution. And when $x_{0}:=y_{0}=z_{0}$, then $x_{0}$ is called the solution of the problem {\rm{\eqref{eqn:Mixed-impul-Hilfer}}}.
\end{definition}
The following theorem guarantees the existence of the extremal mild solution of the impulsive system \eqref{eqn:Mixed-impul-Hilfer}.
\begin{theorem}
\label{Mixed-Th-1}
Let $E$ be an ordered Banach space with the positive normal cone $N$. Assume that $Q(t)\geq 0$ and the impulsive system {\rm{\eqref{eqn:Mixed-impul-Hilfer}}} has both lower and upper solution given by $y_{0}$ and $z_{0}$ respectively, where $y_{0}, z_{0}\in PC_{1-\lambda}$ and $y_{0}\leq z_{0}$. By embracing the mixed monotone iterative procedure and presuming the following assumptions, the impulsive system {\rm{\eqref{eqn:Mixed-impul-Hilfer}}} has extremal solutions between $y_{0}$ and $z_{0}$.

\noindent
{\bf$A(1).$}
There exist constants $C\geq 0$ and $L\leq 0$ such that
\begin{align*}
%\label{Monotone-cdn-1}1
g(t,y_{2},z_{2})-g(t,y_{1},z_{1})\geq -C(y_{2}-y_{1})-L(z_{1}-z_{2})
\end{align*}
and $y_{0}(t)\leq y_{1}(t)\leq y_{2}(t)\leq z_{0}(t)$, $y_{0}(t)\leq z_{2}(t)\leq z_{1}(t)\leq z_{0}(t)$ for any $t\in J$.

\noindent
{\bf$A(2).$}
The impulsive function for $t\in J$ satisfies
\begin{align*}
%\label{Monotone-cdn-2}
\phi_{k}(y_{1},z_{1})\leq \phi_{k}(y_{2},z_{2}), \enspace k=1,2,\ldots,l.
\end{align*}

\noindent
{\bf$A(3).$}
The sequence $\{y_{p}\}\subset [y_{0}(t),z_{0}(t)]$ and $\{z_{p}\}\subset [y_{0}(t),z_{0}(t)]$ for $t\in J$ is respectively increasing and decreasing monotonic sequences. In particular, there exists a constant $L_{1}\geq 0$ such that
\begin{align*}
%\label{Monotone-cdn-3}
\alpha \Big(\{g(t,y_{p},z_{p})\}\Big)\leq L_{1}\Big( \alpha \big(\{y_{p}\}\big)+\alpha \big(\{z_{p}\}\big)\Big),\enspace p=1,2,\ldots,.
\end{align*}

\noindent
{\bf $A(4).$}
Let $y_{p}=\mathcal{G}(y_{p-1},z_{p-1})$, $z_{p}=\mathcal{G}(z_{p-1},y_{p-1})$, $p=1,2,\ldots,$ such that sequence $y_{p}(0)$ and $z_{p}(0)$ are convergent.
\end{theorem}
\begin{proof}
As $C>0$, the problem \eqref{eqn:Mixed-impul-Hilfer} can be presented in the form of problem \eqref{eqn:Mixed-alternate}. So the proof of the existence of a unique mild solution for the problem \eqref{eqn:Mixed-alternate} is sufficient. Define the operator $\mathcal{G}:[y_{0},z_{0}]\times [y_{0},z_{0}]\rightarrow PC_{1-\lambda}(J,E)$ by
\begin{align}
\label{eqn:Mixed-mild-alternate}
\mathcal{G}(y,z)(t)=
\left\{
  \begin{array}{ll}
  S^{*}_{\mu,\nu}(t)x_{0}+\int_{0}^{t}(t-s)^{\mu-1}P^{*}_{\mu}(t-s)\Big[g(s,y(s),z(s))+\\
  \quad \quad (C+L)y(s)-Lz(s)\Big]ds, \enspace t\in [0,t_{1}],\\
  S^{*}_{\mu,\nu}(t)x_{0}+\displaystyle \sum_{i=1}^{k}S^{*}_{\mu,\nu}(t-t_{i})\phi_{i}(y(t_{i}),z(t_{i}))\\
  \quad \quad+\int_{0}^{t}(t-s)^{\mu-1}P^{*}_{\mu}(t-s)\Big[g(s,y(s),z(s))\\
  +(C+L)y(s)-Lz(s)\Big]ds, \quad t\in (t_{k},t_{k+1}],\enspace k=1,2,\ldots l.
  \end{array}
\right.
\end{align}
The map $\mathcal{G}(y,z)(t)$ is continuous since $g$ is continuous. By the Definition \ref{Hilfer-mild-solution}, the fixed points of the operator $\mathcal{G}$ are equivalent to the mild solution of the system given in \eqref{eqn:Mixed-solution-origi}. That is,
\begin{align}
\label{eqn:Mixed-fixed}
\mathcal{G}(x,x)=x.
\end{align}
The following steps are required for the completion of the proof.

\noindent
{\bf\underline{Step 1.}} To show $\mathcal{G}(y_{1},z_{1})\leq \mathcal{G}(y_{2},z_{2})$.

\vspace{3mm} The condition $A(1)$ is used to reduce the below inequalities, which can be applied directly in the proof of the theorem. That is, $\forall t\in J^{'}$,
\begin{align*}
y_{0}(t)\leq y_{1}(t)\leq y_{2}(t)\leq z_{0}(t),& \enspace y_{0}(t)\leq z_{2}(t)\leq z_{1}(t)\leq z_{0}(t). \\
\Rightarrow g(t,y_{1}(t),z_{1}(t))+C y_{1}(t)-Lz_{1}(t)&\leq g(t,y_{2}(t),z_{2}(t))+Cy_{2}(t)-Lz_{2}.
\end{align*}
\begin{align}
\label{eqn:Mixed:alternate-A1}
\Rightarrow g(t,y_{1}(t),z_{1}(t))&+(C+L)y_{1}(t)-Lz_{1}(t)\nonumber\\
&\leq g(t,y_{2}(t),z_{2}(t))+(C+L)y_{2}(t)-Lz_{2}.
\end{align}
Considering the case for $t\in J_{0}^{'}$, for $J_{0}^{'}=[0,t_{1}]$:-

 The operators $S^{*}_{\mu,\nu}(t)$ and $P^{*}_{\mu}(t)$ are positive operators, and hence when the mild solutions are compared, using \eqref{eqn:Mixed:alternate-A1}, the following inequality is obtained.
\begin{align*}
\int_{0}&^{t}(t-s)^{\mu-1}P^{*}_{\mu}(t-s)\Big[g(s,y_{1}(s),z_{1}(s))+(C+L)y_{1}(s)-Lz_{1}(s)\Big]ds\leq\\
&\int_{0}^{t}(t-s)^{\mu-1}P^{*}_{\mu}(t-s)\Big[g(s,y_{2}(s),z_{2}(s))+(C+L)y_{2}(s)-Lz_{2}(s)\Big]ds.
\end{align*}
In which case, for $\forall t\in J_{k}^{'}$, with $J_{k}^{'}=(t_{k},t_{k+1}]$, $k=1,2,\ldots l$, applying the condition $A(2)$  yields
\begin{align*}
 S^{*}&_{\mu,\nu}(t)x_{0}+\displaystyle \sum_{i=1}^{k}S^{*}_{\mu,\nu}(t-t_{i})\phi_{i}(y_{1}(t_{i}),z_{1}(t_{i}))\\
 &+\int_{0}^{t}(t-s)^{\mu-1}P^{*}_{\mu}(t-s)\Big[g(s,y_{1}(s),z_{1}(s))+(C+L)y_{1}(s)-Lz_{1}(s)]ds\leq\\
 S^{*}&_{\mu,\nu}(t)x_{0}+\displaystyle \sum_{i=1}^{k}S^{*}_{\mu,\nu}(t-t_{i})\phi_{i}(y_{2}(t_{i}),z_{2}(t_{i}))\\
 &+\int_{0}^{t}(t-s)^{\mu-1}P^{*}_{\mu}(t-s)\Big[g(s,y_{2}(s),z_{2}(s))+(C+L)y_{2}(s)-Lz_{2}(s)\Big]ds.
\end{align*}
Eventually, $\mathcal{G}(y_{1},z_{1})(t)\leq \mathcal{G}(y_{2},z_{2})(t)$ for $t\in J$.

\noindent
{\bf\underline{Step 2.}} To show $y_{0}\leq\mathcal{G}(y_{0},z_{0})$ ; $\mathcal{G}(z_{0},y_{0})\leq z_{0}$:-

\vspace{3mm} For the case for which $t\in J_{0}^{'}$:-

Let $D^{\mu,\nu}_{0+}z_{0}(t)+Az_{0}(t)+Cz_{0}(t)=\xi(t)$, $\xi(t)\in PC_{1-\lambda}(J,E)$  By the Definition \ref{eqn:Mixed-impul-Hilfer-upper} of the coupled upper $L$-quasi solution, the mild solution of the system \eqref{eqn:Mixed-impul-Hilfer} can be written as
\begin{align*}
z_{0}(t)=&S^{*}_{\mu,\nu}(t)z_{0}+\int_{0}^{t}(t-s)^{\mu-1}P^{*}_{\mu}(t-s)\xi(s)ds\\
\geq & S^{*}_{\mu,\nu}(t)x_{0}+\int_{0}^{t}(t-s)^{\mu-1}P^{*}_{\mu}(t-s)\Big[g(s,z_{0}(s),y_{0}(s))\\
+&(C+L)z_{0}(s)-Ly_{0}(s)\Big]ds.
\end{align*}
From \eqref{eqn:Mixed-mild-alternate}, it can be observed that $z_{0}(t)\geq  \mathcal{G}(z_{0},y_{0})(t).$\\
For $t\in J_{1}^{'}$:-
\begin{align*}
S^{*}&_{\mu,\nu}(t)z_{0}+S^{*}_{\mu,\nu}(t-t_{1})\phi_{1}(z_{0}(t_{1}),y_{0}(t_{1}))+\int_{0}^{t}(t-s)^{\mu-1}P^{*}_{\mu}(t-s)\xi(s)ds\\
&\geq  S^{*}_{\mu,\nu}(t)x_{0}+S^{*}_{\mu,\nu}(t-t_{1})\phi_{1}(z_{0}(t_{1}),y_{0}(t_{1}))+\int_{0}^{t}(t-s)^{\mu-1}P^{*}_{\mu}(t-s)\\
&\Big[g(s,z_{0}(s),y_{0}(s))+(C+L)z_{0}(s)-Ly_{0}(s)\Big]ds.\\
\Rightarrow & z_{0}(t)\geq  \mathcal{G}(z_{0},y_{0})(t).
\end{align*}
Progressing in the same manner, every $J_{k}^{'}$, yields $z_{0}(t)\geq \mathcal{G}(z_{0},y_{0})(t)$. In the same manner, it can be proved that $y_{0}(t)\leq \mathcal{G}(y_{0},z_{0})(t)$ by considering the lower $L$-quasi solutions. Altogether, it can be deduced that
\begin{align*}
y_{0}(t)\leq \mathcal{G}(y_{0},z_{0})(t)\leq \mathcal{G}(x,x)(t)\leq\mathcal{G}(z_{0},y_{0})(t)\leq z_{0}(t).
\end{align*}
Henceforth the conclusion may be drawn that
$$
\mathcal{G}:[y_{0},z_{0}]\times [y_{0},z_{0}] \rightarrow PC_{1-\lambda}(J,E)
$$
is an increasing mixed monotonic operator. By means of the iterative pattern, two sequence $\{y_{p}\}$ and $\{z_{p}\}$ can be defined as,
\begin{align}
\label{eqn:Mixed-iterative-pattern}
y_{p}=\mathcal{G} (y_{p-1},z_{p-1}); \enspace z_{p}=\mathcal{G} (z_{p-1},y_{p-1});\enspace p=1,2,\ldots.
\end{align}
Eventually, due to the monotonicity property of $ \mathcal {G}$, an increasing sequence is derived as,
\begin{align}
\label{eqn:Mixed-inequality}
y_{0}\leq y_{1}\leq y_{2}\leq \ldots \leq y_{p}\leq \ldots \leq z_{p}\leq \ldots \leq z_{2}\leq z_{1}\leq z_{0}.
\end{align}

\noindent
{\bf\underline{Step 3.}} Convergence of sequences $\{y_{p}\}$ and $\{z_{p}\}$ in $J^{'}$:-

Let $B_{p}=\{y_{p}|p\in \mathbb{N}\}+\{z_{p}|p\in \mathbb{N}\}$; $B_{1}=\{y_{p-1}|p\in \mathbb{N}\}$; $B_{2}=\{z_{p-1}|p\in \mathbb{N}\}$; $B_{3}=\{(y_{p-1},z_{p-1})\}|p\in \mathbb{N}$ and $B_{4}=\{(z_{p-1},y_{p-1})\}|p\in \mathbb{N}$. Equation \eqref{eqn:Mixed-iterative-pattern} gives the relation $B_{1}=\mathcal {G}(B_{3}(t))$ and $B_{2}=\mathcal {G}(B_{4}(t))$. Let $\psi (t):=\alpha(B_{p}(t))$. By proving that $\psi(t)\equiv 0$ on every interval $J^{'}_{k}$ means that $\alpha(B_{p}(t))\equiv 0$ for $k=1,2,\ldots,l$, and hence $\{y_{p}\}+\{z_{p}\}$ is precompact in $E$ for every $t\in J$. Ultimately, by the definition of precompact, $\{y_{p}\}$ and $\{z_{p}\}$ have converging subsequence in $E$. Thus it is necessary to prove that $\psi (t)\equiv0$.

For $t\in J_{0}^{'}$ for $J_{0}^{'}=(0,t_{1}]$:-
\begin{align*}
\psi(t)&=\alpha(B_{p}(t))=\alpha \Big(B_{1}(t)+B_{2}(t)\Big)\\
&=\alpha\Big(\mathcal {G}(B_{3}(t))+\mathcal {G}(B_{4}(t))\Big)=\alpha\Big(\mathcal {G}(y_{p-1},z_{p-1})(t)+ \mathcal {G}(z_{p-1},y_{p-1})(t)\Big)
\end{align*}
\begin{align*}
\psi(t)&=\alpha\Big(\Big\{S^{*}_{\mu,\nu}(t)x_{0}+\int_{0}^{t}(t-s)^{\mu-1}P^{*}_{\mu}(t-s)\Big[g(s,y_{p-1},z_{p-1}(s))\\
&+(C+L)y_{p-1}(s)-Lz_{p-1}(s)\Big]ds+S^{*}_{\mu,\nu}(t)x_{0}+\int_{0}^{t}(t-s)^{\mu-1}P^{*}_{\mu}(t-s)\\
&\Big[g(s,z_{p-1}(s),y_{p-1}(s))+(C+L)z_{p-1}(s)-Ly_{p-1}(s)\Big]ds\Big\}:p\in \mathbb{N} \Big).
\end{align*}
The below inequality is the consequence of Lemma \ref{lem:Mixed-cgt}.
\begin{align*}
\psi(t)\leq & 2 \int_{0}^{t}\alpha\Big(\Big\{(t-s)^{\mu-1}P^{*}_{\mu}(t-s)\Big[g(s,y_{p-1}(s),z_{p-1}(s))\\
&+g(s,z_{p-1}(s),y_{p-1}(s))+C\big(y_{p-1}(s)+z_{p-1}(s)\big)\Big]ds\Big\}:p=1,2,\ldots\Big).
\end{align*}
Applying the assumed conditions along with Lemma \ref{lem:Relax-bounds} results in
\begin{align*}
\psi(t) &\leq \dfrac{2M^{*}}{\Gamma(\mu)}\int_{0}^{t}(t-s)^{\mu-1}\Big[(2L_{1}+C)\big(\alpha B_{1}(s)+\alpha B_{2}(s)\big)\Big]ds.\\
&=\dfrac{2M^{*}}{\Gamma(\mu)}(2L_{1}+C)\int_{0}^{t}(t-s)^{\mu-1}\psi(s)ds.
\end{align*}
By Lemma \ref{lem:Gronwall}, $\psi (t)\equiv 0$ on $J_{0}^{'}$. Thus $\{y_{p}(t)\}+\{z_{p}(t)\}$ is precompact and hence $\{y_{p}(t)\}$ and $\{z_{p}(t)\}$ are precompact for $t\in [0,t_{1}]$. In this regard, $\phi_{1}(B_{3}(t_{1}))$ and $\phi_{1}(B_{4}(t_{1}))$ are indeed precompact. Hence

$\alpha\big(\phi_{1}(B_{3}(t_{1}))\big)=0$ and $\alpha\big(\phi_{1}(B_{4}(t_{1}))\big)=0$.

Now for $t\in J_{1}^{'}$, for $J_{1}^{'}=(t_{1},t_{2}]$:-
\begin{align*}
\psi(t)=\alpha (B_{p}(t))&=\alpha\Big(\mathcal {G}(B_{3}(t))+\mathcal {G}(B_{4}(t))\Big)\\
&=\alpha\Big(\mathcal {G}(y_{p-1},z_{p-1})(t)+ \mathcal {G}(z_{p-1},y_{p-1})(t)\Big)
\end{align*}
\begin{align*}
\psi(t)&=\alpha\Bigg(\Big\{S^{*}_{\mu,\nu}(t)x_{0}+S^{*}_{\mu,\nu}(t)\phi_{1}\Big(y_{p-1}(t_{1}),z_{p-1}(t_{1})\Big)
+\int_{0}^{t}(t-s)^{\mu-1}\\
&P^{*}_{\mu}(t-s)\Big[g(s,y_{p-1}(s),z_{p-1}(s))+(C+L)y_{p-1}(s)-Lz_{p-1}(s)\Big]ds\\
&+S^{*}_{\mu,\nu}(t)x_{0}+S^{*}_{\mu,\nu}(t)\phi_{1}\Big(z_{p-1}(t_{1}),y_{p-1}(t_{1})\Big)+\int_{0}^{t}(t-s)^{\mu-1}\\
&P^{*}_{\mu}(t-s)\Big[g(s,z_{p-1}(s),y_{p-1}(s))+(C+L)z_{p-1}(s)-Ly_{p-1}(s)\Big]ds\Big\}\Bigg).
\end{align*}
\begin{align*}
\psi(t)&\leq \dfrac{2M^{*}b^{1-\lambda}}{\Gamma (\lambda)}\Big[\alpha \big(\phi_{1}(B_{3}(t_{1}))+\phi_{1}(B_{4}(t_{1}))\big)\Big]\\
&+\dfrac{2M^{*}}{\Gamma(\mu)}(2L_{1}+C)\int_{0}^{t}\int_{0}^{t}(t-s)^{\mu-1}\psi(s)ds.\\
&\leq \dfrac{2M^{*}}{\Gamma(\mu)}(2L_{1}+C)\int_{0}^{t}(t-s)^{\mu-1}\psi(s)ds.
\end{align*}
By Lemma \ref{lem:Gronwall}, $\psi (t)\equiv 0$ in $J_{1}^{'}$. By Proceeding the same way interval by interval, it can be proved that $\psi(t)\equiv 0$ on every interval $J_{k}^{'}$, $k=1,2,\ldots l$. Thus $\{y_{p}\}$ and $\{z_{p}\}$ are precompact and eventually for $p=1,2,\ldots$, $\{y_{p}\}$ and $\{z_{p}\}$ has a converging subsequence and from \eqref{eqn:Mixed-inequality}, it can be observed that $\{y_{p}\}$ and $\{z_{p}\}$ are converging sequences and hence there exists $\underline{x}(t)$, $\overline{x}(t)\in E$ such that
\begin{align*}
\displaystyle \lim_{p\rightarrow \infty}y_{p}(t)\rightarrow \underline{x}(t),\enspace \displaystyle \lim_{p\rightarrow \infty}z_{p}(t)\rightarrow \overline{x}(t),\enspace t\in J.
\end{align*}
 From \eqref{eqn:Mixed-iterative-pattern} and using the fact that $y_{p}(t)=\mathcal{G}(y_{p-1},z_{p-1})(t)$, \eqref{eqn:Mixed-mild-alternate} can be represented as below
\begin{align*}
y_{p}(t)=
\left\{
  \begin{array}{ll}
  S^{*}_{\mu,\nu}(t)x_{0}+\int_{0}^{t}(t-s)^{\mu-1}P^{*}_{\mu}(t-s)\Big[g(s,y_{p-1}(s),z_{p-1}(s))\\
  \enspace +(C+L)y_{p-1}(s)-Lz_{p-1}(s)\Big]ds, \enspace t\in [0,t_{1}]\\
  S^{*}_{\mu,\nu}(t)x_{0}+\displaystyle \sum_{i=1}^{k}S^{*}_{\mu,\nu}(t-t_{i})\phi_{i}(y_{p-1}(t_{i}),z_{p-1}(t_{i}))\\
  \enspace+\int_{0}^{t}(t-s)^{\mu-1}P^{*}_{\mu}(t-s)\Big[g(s,y_{p-1}(s),z_{p-1}(s))\\
  \enspace +(C+L)y_{p-1}(s)-Lz_{p-1}(s)\Big]ds,\enspace t\in (t_{k},t_{k+1}],\enspace k=1,2,\ldots l.
  \end{array}
\right.
\end{align*}
Using Lebesgue dominated convergence theorem, as $p\rightarrow \infty$
\begin{align*}
\underline{x}(t)=
\left\{
  \begin{array}{ll}
  S^{*}_{\mu,\nu}(t)x_{0}+\int_{0}^{t}(t-s)^{\mu-1}P^{*}_{\mu}(t-s)\Big[g(s,\underline{x}(s),\overline{x}(s))\\
  \enspace +(C+L)\underline{x}(s)-L\overline{x}(s)\Big]ds, \enspace t\in [0,t_{1}],\\
  S^{*}_{\mu,\nu}(t)x_{0}+\displaystyle \sum_{i=1}^{k}S^{*}_{\mu,\nu}(t-t_{i})\phi_{i}(\underline{x}(t_{i}),\overline{x}(t_{i}))\\
  \enspace+\int_{0}^{t}(t-s)^{\mu-1}P^{*}_{\mu}(t-s)\Big[g(s,\underline{x}(s),\overline{x}(s))\\
  \enspace +(C+L)\underline{x}(s)-L\overline{x}(s)\Big]ds,\enspace t\in (t_{k},t_{k+1}],\enspace k=1,2,\ldots l.
  \end{array}
\right.
\end{align*}
It can be observed that $\underline{x}(t)\in PC_{1-\lambda}(J,E)$ and $\underline{x}(t)=\mathcal{G}(\underline{x},\overline{x})(t)$. In a similar manner, it can be proved that $\exists$ $\overline{x}(t)\in PC_{1-\lambda}(J,E)$ such that $\overline{x}(t)=\mathcal{G}(\underline{x},\overline{x})(t)$. With the monotonicity property of $\mathcal{G}$, it can be concluded that $y_{0}\leq \underline{x}\leq \overline{x}\leq z_{0}$. This proves that there exists minimal and maximal solutions $\underline{x}$ and $\overline{x}$ respectively in $[y_{0},z_{0}]$ for the given impulsive system \eqref{eqn:Mixed-impul-Hilfer}.
\end{proof}
%
%\begin{remark}
%The above proved theorem is true for the case when the cone $N$ which is normal is replaced with positive cone which is regular. For detailed proof {\rm{\cite[Corollary 3.3]{Monotone-frac-impulsive}}} may be referred.
%\end{remark}
%
The existence of mild solution of \eqref{eqn:Mixed-impul-Hilfer} can be discussed by replacing the conditions $A(2)$ and $A(3)$ by the below given conditions.

\noindent
{\bf$A(2^{*}).$}
 The impulsive function $\phi_{k}(\cdot,\cdot)$ satisfies
\begin{align*}
\phi_{k}(y_{1},z_{1})\leq \phi_{k}(y_{2},z_{2}), \enspace k=1,2,\ldots,l.
\end{align*}
for $t\in J$ and $y_{0}(t)\leq y_{1}\leq y_{2}\leq z_{0}(t)$, $y_{0}(t)\leq z_{2}\leq z_{1}\leq z_{0}(t)$. Also there exists $M_{k}>0$ satisfying the condition given by
\begin{align}
\label{Mixed-Mk cdn}
\sum_{k=1}^{l}M_{k}\leq \dfrac{\Gamma(\lambda)\Gamma(\mu+1)-4M^{*}(2L_{1}+C)T^{\mu}}{4M^{*}T^{\lambda-1}\Gamma(\mu-1)}
\end{align}
such that $\alpha\big(\phi_{k}(D_{1},D_{2})\big)\leq M_{k}[\alpha(D_{1})+\alpha(D_{2})]$. In the above condition \eqref{Mixed-Mk cdn} the denominator does not jump to infinity as the term $M_{k}$ takes its value from $J_{k}^{'}$, $k>0$. That is $T\neq 0$.

\noindent
{\bf$A(3^{*}).$} There exists a constant $L_{1}<0$ such that
\begin{align*}
\alpha(g(t,D_{1},D_{2}))\leq L_{1}(\alpha(D_{1})+\alpha(D_{2})),\enspace t\in J
\end{align*}

\noindent with the countable sets $D_{1}=\{y_{p}\}$ and $D_{2}=\{z_{p}\}$ in $[y_{0}(t),z_{0}(t)]$. The following theorem utilizes both the above inequalities where the existence of atleast one mild solution between the coupled $L$-quasi upper and lower solutions is investigated. Also the semigroup generated by the operator $-A$ is assumed to be equicontinuous.
\begin{theorem}
\label{Mixed:Th-2}
Let $E$ be an ordered Banach space with the positive normal cone $N$. Assume that $Q(t)\geq 0$ which is equicontinuous on $E$; $g\in C(J \times E \times E,E)$; $x_{0}\in E$. Let the impulsive system {\rm{\eqref{eqn:Mixed-impul-Hilfer}}} has coupled $L$-quasi lower and upper solution, given by $y_{0}$ and $z_{0}$ respectively, where $y_{0}, z_{0}\in PC_{1-\lambda}$ and $y_{0}\leq z_{0}$. If the assumptions $A(1)$, $A(2^{*})$ and $A(3^{*})$ are satisfied, then the impulsive system {\rm{\eqref{eqn:Mixed-impul-Hilfer}}} has coupled minimal and maximal $L$-quasi mild solution between $[y_{0},z_{0}]$ and at least one mild solution in $[y_{0},z_{0}]$ between $\underline{x}$ and $\overline{x}$ such that for $p\rightarrow\infty$, $y_{p}(t)\rightarrow \underline{x}; \enspace z_{p}(t)\rightarrow \overline{x}, \enspace t\in J$. Here $y_{p}$ and $z_{p}$ are given as $y_{p}=\mathcal{G}(y_{p-1},z_{p-1})$, $z_{p}=\mathcal{G}(z_{p-1},y_{p-1})$, that satisfy
\begin{align*}
y_{0}(t)\leq y_{1}(t)\leq \ldots \leq y_{p}(t)\leq \ldots \leq \underline{x} \leq \overline{x} \leq \ldots \leq z_{p}(t)\ldots \leq z_{1}(t)\leq z_{0}(t).
\end{align*}
\end{theorem}
\begin{proof}

It can be verified that the assumption $A(3^{*})\Rightarrow A(3)$. Hence by the theorem \eqref{Mixed-Th-1}, the impulsive system \eqref{eqn:Mixed-impul-Hilfer} has minimal and maximal $L$-quasi lower $\underline{x}$ and upper $\overline{x}$ solutions in $[y_{0},z_{0}]$. From the normality definition of the cone $P$, there exists $\tilde{M}>0$ such that
\begin{align}
\label{eqn:equi-inequality}
\|g(t,y(t),z(t))+(c+L)y(t)-Lz(t)\|\leq \tilde{M}.
\end{align}
The proof of the theorem terminates in finding at least one mild solution in $[y_{0},z_{0}]$. First, let the operator $\mathcal{F}$ be defined as $\mathcal{F}:[y_{0},z_{0}]\rightarrow [y_{0},z_{0}]$ such that $\mathcal{F}x=\mathcal{G}(x,x)$. It is evident that $\mathcal{F}$ is continuous and the fixed point of the operator $\mathcal{F}$ is equivalent to the mild solution of the system \eqref{eqn:Mixed-impul-Hilfer}.

\vspace{3mm} For the case for which $t\in J_{0}^{'}$:-

 Let $s_{1},s_{2}\in[0,t_{1}]$ such that $0<s_{1}<s_{2}\leq t_{1}$. The following inequality determines the equicontinuous of the operator $\mathcal{F}$.
\begin{align*}
\Big\|s_{2}^{1-\lambda}&(\mathcal{F}x)(s_{2})-s_{1}^{1-\lambda}(\mathcal{F}x)(s_{2})\Big\|\leq \Big\|s_{2}^{1-\lambda}\mathcal{G}(x,x)(s_{2})-s_{1}^{1-\lambda}\mathcal{G}(x,x)(s_{2})\Big\|\\
&\leq\Big\|s_{2}^{1-\lambda}S_{\mu,\nu}^{*}(s_{2})x_{0}-s_{1}^{1-\lambda}S_{\mu,\nu}^{*}(s_{1})x_{0}\Big\|+\Big\|\int_{0}^{s_{2}}s_{2}^{1-\lambda}(s_{2}-s)^{\mu-1}\\
&P^{*}_{\mu}(s_{2}-s)\big[g(s,x(s),x(s))+(C+L)x(s)-Lx(s)\big]ds\Big\|\\
&-\Big\|\int_{0}^{s_{1}}s_{1}^{1-\lambda}(s_{1}-s)^{\mu-1}P^{*}_{\mu}(s_{1}-s)\big[g(s,x(s),x(s))\\
&+(C+L)x(s)-Lx(s)\big]ds\Big\|
\end{align*}
For convenience let $g(s,x(s),x(s))+Cx(s)$ be denoted by $\zeta(s)$.
\begin{align*}
\Big\|s_{2}^{1-\lambda}&(\mathcal{F}x)(s_{2})-s_{1}^{1-\lambda}(\mathcal{F}x)(s_{2})\Big\|\leq \Bigg(\Big\|s_{2}^{1-\lambda}S_{\mu,\nu}^{*}(s_{2})x_{0}-s_{2}^{1-\lambda}S_{\mu,\nu}^{*}(s_{1})x_{0}\Big\|\\
&+\Big\|s_{2}^{1-\lambda}S_{\mu,\nu}^{*}(s_{1})x_{0}-s_{1}^{1-\lambda}S_{\mu,\nu}^{*}(s_{1})x_{0}\Big\|\Bigg)\\
&+\Big\|s_{2}^{1-\lambda}\int_{s_{1}}^{s_{2}}(s_{2}-s)^{\mu-1}P_{\mu}^{*}(s_{2}-s)\big[\zeta(s)\big]ds\Big\|\\
&+\Big\|\int_{0}^{s_{1}}\big(s_{2}^{1-\lambda}(s_{2}-s)^{\mu-1}-s_{1}^{1-\lambda}(s_{1}-s)^{\mu-1}\big)P_{\mu}^{*}(s_{2}-s)\big[\zeta(s))\big]ds\Big\|\\
&+\Big\|s_{1}^{1-\lambda}\int_{0}^{s_{1}}(s_{1}-s)^{\mu-1}\big(P_{\mu}^{*}(s_{2}-s)-P_{\mu}^{*}(s_{1}-s)\big)\big[\zeta(s)\big]ds\Big\|.\\
&= \displaystyle\sum_{i=1}^{5}\|I_{i}\|.
\end{align*}
Now for $i=1,2,\ldots,5$, $I_{i}$ can be calculated individually as below.
For $I_{1}$, using Remark \eqref{rem:Mixed-Remark-alternate}, it can be observed that
\begin{align*}
I_{1}&=\Big\|s_{2}^{1-\lambda}S_{\mu,\nu}^{*}(s_{2})x_{0}-s_{2}^{1-\lambda}S_{\mu,\nu}^{*}(s_{1})x_{0}\Big\|\\
&\leq\Big\|s_{2}^{1-\lambda}\big(S_{\mu,\nu}^{*}(s_{2})-S_{\mu,\nu}^{*}(s_{1})\big)\Big\|\|x_{0}\|\\
&\longrightarrow 0, \enspace \mbox{as}\enspace s_{2}\rightarrow s_{1}.
\end{align*}
For $I_{2}$, using Remark \eqref{rem:Mixed-Remark-alternate}, the following observation can be made similar to $I_{1}$.
\begin{align*}
I_{2}&=\Big\|s_{2}^{1-\lambda}S_{\mu,\nu}^{*}(s_{1})x_{0}-s_{1}^{1-\lambda}S_{\mu,\nu}^{*}(s_{1})x_{0}\Big\|\\
&\leq \dfrac{M^{*}T^{\lambda-1}}{\Gamma(\lambda)}\|s_{2}^{1-\lambda}-s_{1}^{1-\lambda}\|\|x_{0}\|
=\dfrac{M^{*}T^{\lambda-1}}{\Gamma(\lambda)}\|\big(s_{2}-s_{1}\big)^{1-\lambda}\|\|x_{0}\|\\
&\longrightarrow 0, \enspace \mbox{as}\enspace s_{2}\rightarrow s_{1}.
\end{align*}
$I_{3}$ can be evaluated using Remark \eqref{rem:Mixed-Remark-alternate} as below.
\begin{align*}
I_{3}&=\Big\|s_{2}^{1-\lambda}\int_{s_{1}}^{s_{2}}(s_{2}-s)^{\mu-1}P_{\mu}^{*}(s_{2}-s)\big[\zeta(s)\big]ds\Big\|\\
&\leq \dfrac{M^{*}\tilde{M}}{\Gamma(\mu)}\Big\|\int_{s_{1}}^{s_{2}}(s_{2}-s)^{\mu-1}ds\Big\|\\
&\longrightarrow 0, \enspace \mbox{as}\enspace s_{2}\rightarrow s_{1}.
\end{align*}
$I_{4}$ is evaluated using Remark \eqref{rem:Mixed-Remark-alternate} and Equation \eqref{eqn:equi-inequality}, as below.
\begin{align*}
I_{4}=&\Big\|\int_{0}^{s_{1}}\big(s_{2}^{1-\lambda}(s_{2}-s)^{\mu-1}-s_{1}^{1-\lambda}(s_{1}-s)^{\mu-1}\big)P_{\mu}^{*}(s_{2}-s)\big[\zeta(s)\big]ds\Big\|\\
&\Rightarrow I_{4}\leq \dfrac{M^{*}\tilde{M}}{\Gamma(\mu)}\Big\|\int_{0}^{s_{1}}\big(s_{2}^{1-\lambda}(s_{2}-s)^{\mu-1}-s_{1}^{1-\lambda}(s_{1}-s)^{\mu-1}\big)ds\Big\|\\
&\longrightarrow 0, \enspace \mbox{as}\enspace s_{2}\rightarrow s_{1}.
\end{align*}
Similarly for $\epsilon \in (0,s_{1})$, $I_{5}$ can be evaluated as below.
\begin{align*}
I_{5}&=\Big\|\int_{0}^{s_{1}-\epsilon}s_{1}^{1-\lambda}(s_{1}-s)^{\mu-1}\big(P_{\mu}^{*}(s_{2}-s)-P_{\mu}^{*}(s_{1}-s)\big)\big[\zeta(s)\big]ds\Big\|\\
&+\Big\|\int_{s_{1}-\epsilon}^{s_{1}}s_{1}^{1-\lambda}(s_{1}-s)^{\mu-1}\big(P_{\mu}^{*}(s_{2}-s)-P_{\mu}^{*}(s_{1}-s)\big)\big[\zeta(s)\big]ds\Big\|
\end{align*}
\begin{align*}
\Rightarrow I_{5}&\leq \tilde{M}\int_{0}^{s_{1}-\epsilon}s_{1}^{1-\lambda}(s_{1}-s)^{\mu-1}\sup_{s\in[0,s-\epsilon]}\Big\|\big(P_{\mu}^{*}(s_{2}-s)-P_{\mu}^{*}(s_{1}-s)\big)\Big\|ds\\
&+\dfrac{2\tilde{M} M^{*}}{\Gamma(\mu)}\int_{s_{1}-\epsilon}^{s_{1}}s_{1}^{1-\lambda}(s_{1}-s)^{\mu-1}ds.\\
&\leq \tilde{M}\int_{0}^{s_{1}-\epsilon}s_{1}^{1-\lambda}s^{\mu-1}\sup_{s\in[0,s-\epsilon]}\Big\|\big(P_{\mu}^{*}(s_{2}+s-s_{1})-P_{\mu}^{*}(s)\big)\Big\|ds\\
&+\dfrac{2\tilde{M} M^{*}t_{1}^{1-\lambda}\epsilon^{\mu}}{\Gamma(\mu+1)}\\
&\longrightarrow 0, \enspace \mbox{as}\enspace \epsilon\rightarrow 0\enspace \mbox{and} \enspace s_{2}\rightarrow s_{1}.
\end{align*}
Thus the following conclusion can be drawn for $J_{0}^{'}$.
\begin{align*}
\implies &\Big\|s_{2}^{1-\lambda}(\mathcal{F}x)(s_{2})-s_{1}^{1-\lambda}(\mathcal{F}x)(s_{2})\Big\|\longrightarrow 0.
\end{align*}

For $J_{k}^{'}=(t_{k},t_{k+1}]$, let $s_{1},s_{2}\in (t_{k},t_{k+1}]$ such that $t_{k}<s_{1}<s_{2}\leq t_{k+1}$, for which the following equality is evaluated.
\begin{align*}
\Big\|(s_{2}-t_{k})&^{1-\lambda}(\mathcal{F}x)(s_{2})-(s_{1}-t_{k})^{1-\lambda}(\mathcal{F}x)(s_{2})\Big\|\\
&=\Big\|(s_{2}-t_{k})^{1-\lambda}\mathcal{G}(x,x)(s_{2})-(s_{1}-t_{k})^{1-\lambda}\mathcal{G}(x,x)(s_{2})\Big\|.
\end{align*}
Calculations similar to $J_{0}^{'}$ are performed to obtain the following observation.
\begin{align*}
\Big\|(s_{2}-t_{k})^{1-\lambda}\mathcal{G}(x,x)(s_{2})-(s_{1}-t_{k})^{1-\lambda}\mathcal{G}(x,x)(s_{2})\Big\|\rightarrow 0, \enspace \mbox{as}\enspace s_{2}\rightarrow s_{1}\\
\implies \Big\|\mathcal{G}(x,x)(s_{2})-\mathcal{G}(x,x)(s_{2})\Big\|\rightarrow 0, \enspace \mbox{as}\enspace s_{2}\rightarrow s_{1}.
\end{align*}
Consequently, $\Big\|(\mathcal{F}x)(s_{2})-(\mathcal{F}x)(s_{2})\Big\|\rightarrow 0$ independently of $x\in [y_{0},z_{0}]$ as $s_{2}\rightarrow s_{1}$, which implies that
$(\mathcal{F}x):[y_{0},z_{0}]\rightarrow [y_{0},z_{0}]$ is equicontinuous. In this regard, for any $D\subset [y_{0},z_{0}]$, $\mathcal{F}(D)\subset [y_{0},z_{0}]$ is bounded and equicontinuous. By Lemma \ref{lem:Mixed lem-1} it is evident that there exists a countable set $D_{0}=\{x_{p}\}\subset D$, such that
\begin{align*}
\alpha(\mathcal{F}(D))\leq 2 \alpha (\mathcal{F}(D_{0})).
\end{align*}
From Lemma \ref{lem:Mixed-lem-3}, it can be observed that
\begin{align*}
\alpha(\mathcal{F}(D_{0}))=\max_{t\in J}\alpha(\mathcal{F}(D_{0})(t)).
\end{align*}
For $t\in J_{0}^{'}$, by Lemma \ref{lem:Mixed-cgt}, Equation \eqref{eqn:Mixed-mild-alternate}, and from the assumption $A(3^{*})$, the following inequality is evaluated.
\begin{align*}
\alpha(\mathcal{F}(D_{0})(t))=&\alpha \Big(\Big\{S^{*}_{\mu,\nu}(t)x_{0}+\int_{0}^{t}(t-s)^{\mu-1}P^{*}_{\mu}(t-s)\Big[g\big(s,x_{p}(s),x_{p}(s)\big)
\\
&+Cx_{p}(s)\Big]ds\Big\}\Big)\\
\leq &\alpha \Big(\Big\{S^{*}_{\mu,\nu}(t)x_{0}\Big\}\Big)+\dfrac{2M^{*}}{\Gamma(\mu)}\int_{0}^{t}(t-s)^{\mu-1}\alpha\Big(\Big\{g\big(s,D_{0}(s),D_{0}(s)\big)\\
&+CD_{0}(s)\Big\}\Big)ds\\
\leq &\dfrac{2M^{*}(2L_{1}+C)}{\Gamma(\mu)}\int_{0}^{t}(t-s)^{\mu-1}\alpha\Big(D_{0}(s)\Big)ds.\\
\leq & \dfrac{2M^{*}(2L_{1}+C)T^{\mu}}{\Gamma(\mu+1)}\alpha(D).
\end{align*}
For the case $J_{1}^{'}$, $t\in (t_{1},t_{2}]$, the inequality is evaluated as below using the Lemma \ref{lem:Mixed-cgt}, Equation \eqref{eqn:Mixed-mild-alternate}, and the assumptions $A(2^{*})$ and $A(3^{*})$
\begin{align*}
\alpha(\mathcal{F}(D_{0})(t))=&\alpha \Big(\Big\{S^{*}_{\mu,\nu}(t)x_{0}+\displaystyle S^{*}_{\mu,\nu}(t-t_{1})\phi_{1}(x_{p}(t_{1}),x_{p}(t_{1}))\\
  +&\int_{0}^{t}(t-s)^{\mu-1}P^{*}_{\mu}(t-s)\Big[g(s,x_{p}(s),x_{p}(s))+Cx_{p}(s)\Big]ds\Big\}\Big)\\
  \leq &\alpha \Big(\Big\{S^{*}_{\mu,\nu}(t)x_{0}\Big\}\Big)+\dfrac{M^{*}T^{\lambda-1}}{\Gamma(\lambda)}\alpha \Big(\Big\{\phi_{1}(D_{0}(t_{1}),D_{0}(t_{1}))\Big\}\Big)\\
  +&\dfrac{2M^{*}(2L_{1}+C)}{\Gamma(\mu)}\int_{0}^{t}(t-s)^{\mu-1}\alpha\Big(D_{0}(s)\Big)ds\\
  \leq &2M^{*}\Big(\dfrac{M_{1}T^{\lambda-1}}{\Gamma(\lambda)}+\dfrac{(2L_{1}+C)T^{\mu}}{\Gamma(\mu+1)}\Big)\alpha(D).
\end{align*}
Now for a general case, that is for $J_{k}^{'}$ , $t\in (t_{k},t_{k+1}]$, $k=1,2,\ldots,l$ the inequality is calculated below using the Lemma \ref{lem:Mixed-cgt}, Equation \eqref{eqn:Mixed-mild-alternate}, and the assumptions $A(2^{*})$ and $A(3^{*})$.
\begin{align*}
\alpha(\mathcal{F}(D_{0})(t))=&\alpha \Big(\Big\{S^{*}_{\mu,\nu}(t)x_{0}+\displaystyle\sum_{i=1}^{k} S^{*}_{\mu,\nu}(t-t_{i})\phi_{i}(x_{p}(t_{i}),x_{p}(t_{i}))\\
  +&\int_{0}^{t}(t-s)^{\mu-1}P^{*}_{\mu}(t-s)\Big[g(s,x_{p}(s),x_{p}(s))+Cx_{p}(s)\Big]ds\Big\}\Big)\\
  \leq &\alpha \Big(\Big\{S^{*}_{\mu,\nu}(t)x_{0}\Big\}\Big)+\dfrac{M^{*}T^{\lambda-1}}{\Gamma(\lambda)}\alpha \Big(\Big\{\sum_{i=1}^{k}\phi_{1}(D_{0}(t_{i}),D_{0}(t_{i}))\Big\}\Big)\\
  +&\dfrac{2M^{*}(2L_{1}+C)}{\Gamma(\mu)}\int_{0}^{t}(t-s)^{\mu-1}\alpha\Big(D_{0}(s)\Big)ds\\
  \leq &2M^{*}\Big(\dfrac{\sum_{i=1}^{k}M_{i}T^{\lambda-1}}{\Gamma(\lambda)}+\dfrac{(2L_{1}+C)T^{\mu}}{\Gamma(\mu+1)}\Big)\alpha(D).
\end{align*}
By Lemma \ref{lem:Mixed-lem-3}, since $\mathcal{F}(D_{0})$ is bounded and equicontinuous, the above inequality results in,
\begin{align*}
\alpha(\mathcal{F}(D_{0})(t))\leq &4M^{*}\Big(\dfrac{\sum_{i=1}^{k}M_{i}T^{\lambda-1}}{\Gamma(\lambda)}+\dfrac{(2L_{1}+C)T^{\mu}}{\Gamma(\mu+1)}\Big)\alpha(D)\leq \alpha(D).
\end{align*}
Let $4M^{*}\Big(\dfrac{\sum_{i=1}^{k}M_{i}T^{\lambda-1}}{\Gamma(\lambda)}+\dfrac{(2L_{1}+C)T^{\mu}}{\Gamma(\mu+1)}\Big)=\eta$
\begin{enumerate}
\item
If $\eta<1 $, then the operator $\mathcal{F}:[y_{0},z_{0}]\rightarrow [y_{0},z_{0}]$ is condensing according to Lemma \ref{lem:Sadovskii}. Hence $\mathcal{F}$ has a fixed point $x\in [y_{0},z_{0}]$.
\item
If $\eta\geq 1$, that is, when $\eta$ jumps at the impulsive points, it is necessary to divide the interval $[0,T]$ into $n$ parts such that $\Delta_{n}=0=\tilde{t}_{0}<\tilde{t}_{1}<\ldots<\tilde{t}_{n}=T$. Here the points $\tilde{t}_{0},\tilde{t}_{1},\ldots,\tilde{t}_{n}$ are not the impulse points such that the below condition holds.
\begin{align*}
4M^{*}\Big(\dfrac{\sum_{i=1}^{k}M_{i}\|\Delta_{n}\|^{\lambda-1}}{\Gamma(\lambda)}+\dfrac{(2L_{1}+C)\|\Delta_{n}\|^{\mu}}{\Gamma(\mu+1)}\Big).
\end{align*}
\end{enumerate}
In the interval $[0,\tilde{t}_{1}]$, according to the above statements $(i)$ and $(ii)$, there exists a mild solution $x_{1}(t)\in[0,\tilde{t}_{1}]$. Now in the interval
$[\tilde{t}_{1},\tilde{t}_{2}]$ with initial condition $x(\tilde{t}_{1})=x_{1}(\tilde{t}_{1})$, it has a mild solution $x_{2}(t)\in [\tilde{t}_{1},\tilde{t}_{2}]$. Thus, the mild solution of the equation is extended from $[0,\tilde{t}_{1}]$ to $[0,\tilde{t}_{2}]$. Subsequently, continuing this process, the mild solution of the equation is extended to $[0,T]$. Thus the impulsive system \eqref{eqn:Mixed-impul-Hilfer} has a mild solution $x\in PC_{1-\lambda}(J,E)$ that satisfies $x(t)=x_{i}(t)$ such that $\tilde{t}_{i-1}\leq t \leq \tilde{t}_{i}$, for $i=1,2,\ldots,n$.

Since $x=\mathcal{F}x=\mathcal{G}(x,x)$ for $y_{0}\leq x \leq z_{0}$, with respect to the mixed monotone property, the conclusion can be drawn as  $y_{1}=\mathcal{G}(y_{0},z_{0})\leq \mathcal{G}(x,x) \leq \mathcal{G}(z_{0},y_{0})=z_{1}$. In a similar way, it is true for $y_{2}\leq x \leq z_{2}$ and in general, $y_{p}\leq x\leq z_{p}$. It is clear that, letting $p\rightarrow \infty$ reduces to $\underline{x}\leq x\leq \overline{x}$. Hence it can be concluded that the impulsive system \eqref{eqn:Mixed-impul-Hilfer} has at least one mild solution between $\underline{x}$ and $\overline{x}$.
 \end{proof}
\begin{corollary}
\label{Coro:Mixed-leq}
In an ordered Banach space $E$, let $N$ be the positive cone with normal constant $\tilde{N}$. With the assumption that the operator $Q(t)$ is positive for $t\in J$, if the assumptions $A(1)$ and $A(2)$ are satisfied combined with the condition given below, then the condition $A(3)$ is automatically true.

\noindent
{\bf$A(5).$} There exists a constant $C^{*}$ and $L^{*}$ such that
\begin{align*}
%\label{Monotone-cdn-4}
g(t,y_{2},z_{2})-g(t,y_{1},z_{1})\leq C^{*}(y_{2}-y_{1})+L^{*}(z_{1}-z_{2})
\end{align*}
and $y_{0}(t)\leq y_{1}(t)\leq y_{2}(t)\leq z_{0}(t)$, $y_{0}(t)\leq z_{2}(t)\leq z_{1}(t)\leq z_{0}(t)$ for any $t\in J$.
\end{corollary}
\begin{proof}
Let $\{y_{p}\}$,$\{y_{q}\}$ and $\{z_{p}\}$,$\{z_{q}\}$ be two set of increasing sequences such that
$$
\{y_{p}\},\{y_{q}\},\{z_{p}\},\{z_{q}\} \subset [y_{0}(t),z_{0}(t)],
$$
for $t\in J$ and $p\leq q$.
By the condition $A(1)$ and $A(5)$,
\begin{align*}
\theta \leq g(t,y_{q},z_{q})&-g(t,y_{p},z_{p})+C(y_{q}-y_{p})+L(z_{p}-z_{q})\\
&\leq (C^{*}+C)(y_{q}-y_{p})+(L^{*}+L)(z_{p}-z_{q}).
\end{align*}
From the definition of the normal cone with the normality constant $\tilde{N}$ of the positive cone $N$, the equation further reduces to,
\begin{align*}
\|g(t,y_{q},z_{q})&-g(t,y_{p},z_{p})+C(y_{q}-y_{p})+L(z_{p}-z_{q})\| \\
&\leq \tilde{N}\big((C^{*}+C)(y_{q}-y_{p})+(L^{*}+L)(z_{p}-z_{q})\big).\\
\Rightarrow\|g(t,y_{q},z_{q})&-g(t,y_{p},z_{p})\|\\
&\leq (\tilde{N}C^{*}+\tilde{N}C+C)\|y_{q}-y_{p}\|+(\tilde{N}L^{*}+\tilde{N}L+L)\|z_{p}-z_{q}\|.
\end{align*}
Let $\displaystyle L_{1}=\tilde{N}(C^{*}+C+L^{*}+L)+C+L$. By the definition of measure of non-compactness the above equation reduces to,
\begin{align*}
\alpha \Big(\{g(t,y_{p},z_{p})\}\Big)\leq L_{1}\Big( \alpha \big(\{y_{p}\}\big)+\alpha \big(\{z_{p}\}\big)\Big),\enspace p=1,2,\ldots,.
\end{align*}
Thus the condition $A(3)$ is reduced.
\end{proof}
\begin{theorem}
\label{Mixed:Th-3}
An impulsive fractional system {\rm{(\ref{eqn:Mixed-impul-Hilfer})}} is said to have an unique mild solution that lie between $[y_{0},z_{0}]$, where $y_{0}\in PC_{1-\lambda}$ and $z_{0}\in PC_{1-\lambda}$ are the coupled $L$-quasi lower and upper solution with $y_{0}\leq z_{0}$, if the conditions $A(1)$, $A(2)$, $A(4)$ and $A(5)$ holds.
\end{theorem}
\begin{proof}
If $\overline{x}$ and $\underline{x}$ are the maximal and the minimal solution of the impulsive system (\ref{eqn:Mixed-impul-Hilfer}), then to prove the uniqueness, it has to be proved that $\overline{x}=\underline{x}$. Let $t\in J_{0}^{'}$. Using (\ref{eqn:Mixed-fixed}) in both the solutions results in,
\begin{eqnarray*}
\theta &\leq &\enspace \overline{x}(t)-\underline{x}(t)=\mathcal{G}(\overline{x},\underline{x})(t)-\mathcal{G}(\underline{x},\overline{x})(t)\\
&=& \int_{0}^{t}(t-s)^{\mu-1}P^{*}_{\mu}(t-s)\big[(g(s,\overline{x}(s),\underline{x}(s))- g(s,\underline{x}(s),\overline{x}(s)))\\
&+&(C+2L)(\overline{x}(s)-\underline{x}(s))\big]ds\\
&\leq & \int_{0}^{t}(t-s)^{\mu-1}P^{*}_{\mu}(t-s)[(C^{*}+L^{*})(\overline{x}(t)-\underline{x}(t))\\
&+&(C+2L)(\overline{x}(t)-\underline{x}(t))]ds.
\end{eqnarray*}
Using the normality of the positive cone $N$, the above inequality reduces to,
$$
\|\overline{x}(t)-\underline{x}(t)\|\leq \frac{\tilde{M}M^{*}(C^{*}+L^{*}+C+2L)}{\Gamma (\mu)}\int_{0}^{t}(t-s)^{\mu-1}\|\overline{x}(t)-\underline{x}(t)\|ds.
$$
By Gronwall inequality, $\|\overline{x}(t)-\underline{x}(t)\|=0$. Which implies $\overline{x}(t)=\underline{x}(t)$.

For every interval $J^{'}_{k}$, as $\phi_{k}(\overline{x}(t_{k}),\underline{x}(t_{k}))=\phi_{k}(\overline{x}(t_{k}),\underline{x}(t_{k}))$, the calculation is similar and it results in $\overline{x}(t)=\underline{x}(t)$ for $t\in J_{k}^{'}$, for $k=1,2,\ldots,l$. The uniqueness is thus proved.
\end{proof}

\section{Observation}
\label{discussion}
This paper is based on finding extremal solution of impulsive system with Hilfer fractional derivative using mixed monotone iterative technique. Theorem \ref{Mixed-Th-1} guarantees the existence of minimal and maximal solution for the considered system and Theorem \ref{Mixed:Th-2} discusses the condition such that there exists atleast one mild solution between the minimal and maximal solution. Finally, Theorem \ref{Mixed:Th-3} ensure the uniqueness of such mild solution. The results are proved considering that the semigroup generated by the operator is a non-compact semigroup and is an equicontinuous semigroup.

 The results can further be extended in studying the case when the semigroup generated is compact and for the case when the coupled upper and lower quasi solution does not exist. Also, this article can lead to the study of impulsive system with nonlocal conditions. It is to be noted that the Theorem \ref{Mixed-Th-1} is true for the case when the cone $N$ which is normal is replaced with positive cone which is regular. For detailed proof {\rm{\cite[Corollary 3.3]{Monotone-frac-impulsive}}} may be referred.

\section{Example}
\label{Example}
An example is provided in this section which illustrates the main results.
\begin{example}
\label{Mixed-Example}
 Let $E=L^{p}(\Lambda)$ for $1<p<\infty$ be generated by a positive cone $N$ defined as $N=\{x\in L^{p}(\Lambda):x(y)\geq \theta, \enspace \mbox{a.e}\enspace y\in \Lambda\}$, where $\theta$ is the zero element. Here $\Lambda \subset R^\mathbb{N}$, $\mathbb{N}\geq 1$ is a bounded domain with a sufficiently smooth boundary $\partial\Lambda$. An impulsive Hilfer fractional parabolic partial differential equation with above conditions is considered as below.
\begin{align}
\label{eqn:Mixed-Example}
\left\{
  \begin{array}{ll}
   D_{0+}^{\mu,\nu}x(t,w)-\nabla^{2}x(t,w)= g(t,w,x(t,w),x(t,w)), \enspace (t,y) \in J \times \Lambda \\
   \Delta I_{t_{k}}^{(1-\lambda)}x(t_{k})=\phi_{k}(x(w,t_{k}),x(w,t_{k})),\enspace k=1,2,\ldots l, \enspace y\in \Lambda\\
    I_{0+}^{(1-\lambda)}[x(0,w)]= x_{0}
      \end{array}
\right.
\end{align}
where $D_{0+}^{\mu,\nu}$ is the Hilfer fractional derivative with order $0<\mu<1$ and type $0\leq \nu \leq 1$, $t\in [0,T]$, $\nabla^{2}$ is the Laplace operator such that $-Ax=\nabla^{2}x$, $J=[0,T]$ with impulsive points at $t_{k}$ for $k=0,1,\ldots,l$ such that $J^{'}=J/\{0,t_{1},t_{2},\ldots,t_{l}\}$. Let $-A$ generates a equicontinuous analytic semigroup $Q(t)$ for $t\geq0$ and it is defined as $A:D(A)\subset E\rightarrow E$. Here, $D(A)=W^{2,p}\cap W^{1,p}_{0}(\Lambda)$. The continuous function $g$ is defined as $g:J\times \Lambda \times E \times E \rightarrow E$ and the impulsive function is defined as $\phi_{k}:E \times E \rightarrow E$.

Now, the Example \ref{eqn:Mixed-Example} can be given as an abstract form similar to \eqref{eqn:Mixed-impul-Hilfer}.
\begin{theorem}
Let the  Hilfer fractional system given in Example \ref{Mixed-Example} satisfy the following conditions with $x_{0}\geq 0$.

\noindent
{\bf$E(1).$} Let There exists a function $z=z(t,w)\in PC_{1-\lambda}(J,\Lambda)$ such that
\begin{align*}
\left\{
  \begin{array}{ll}
   D_{0+}^{\mu,\nu}z(t,w)-\nabla^{2}z(t,w)\geq g(t,w,z(t,w),z(t,w)), \enspace t \in J,\\
     \Delta I_{t_{k}}^{(1-\lambda)}z(t_{k},w)\geq \phi_{k}(z(t_{k},w),z(t_{k},w)),\enspace k=1,2,\ldots l\\
    I_{0+}^{(1-\lambda)}[z(0,w)]\geq x_{0}.
      \end{array}
\right.
\end{align*}

\noindent
{\bf$E(2).$}
There exist constants $C\geq 0$ and $L\leq 0$ such that
\begin{align*}
%\label{Monotone-cdn-1}1
g(t,w,y_{2}(t,w),z_{2}(t,w))&-g(t,w,y_{1}(t,w),z_{1}(t,w))\\
&\geq -C(y_{2}(t,w)-y_{1}(t,w))-L(z_{1}(t,w)-z_{2}(t,w))
\end{align*}
and $y_{0}(t,w)\leq y_{1}(t,w)\leq y_{2}(t,w)\leq z_{0}(t,w)$, $y_{0}(t,w)\leq z_{2}(t,w)\leq z_{1}(t,w)\leq z_{0}(t,w)$ for any $t\in J$.

\noindent
{\bf$E(3).$}
The impulsive function for $t\in J$ satisfies
\begin{align*}
%\label{Monotone-cdn-2}
\phi_{k}(y_{1}(t_{k},w),z_{1}(t_{k},w))\leq \phi_{k}(y_{2}(t_{k},w),z_{2}(t_{k},w)), \enspace k=1,2,\ldots,l.
\end{align*}

\noindent
{\bf$E(4).$}
For $t\in J$, the sequence $\{y_{p}(t,w)\}\subset [y_{0}(t,w),z_{0}(t,w)]$ is an increasing monotonic sequence and $\{z_{p}(t,w)\}\subset [y_{0}(t,w),z_{0}(t,w)]$  is a decreasing monotonic sequences. In particular, there exists a constant $L_{1}\geq 0$ such that for $p=1,\ldots,$
\begin{align*}
%\label{Monotone-cdn-3}
\alpha \Big(\{g(t,y_{p}(t,w),z_{p}(t,w))\}\Big)\leq L_{1}\Big( \alpha \big(\{y_{p}(t,w)\}\big)+\alpha \big(\{z_{p}(t,w)\}\big)\Big).
\end{align*}
Then using the monotone iterative procedure initiating from $0$ to $z(t,w)$, the system \eqref{eqn:Mixed-Example} has minimal and maximal solutions.

\begin{proof}
From the assumption $E(1)$, it can be concluded that the lower and upper solution lies between $0$ and $z(t,w)$. Also the Example \ref{Mixed-Example} satisfies all the assumptions of Theorem \ref{Mixed:Th-3}, it can be concluded that there exists a unique solution between $0$ and $z(t,w)$.
\end{proof}

\end{theorem}
\end{example}
\section*{Acknowledgement}
The work of the first author is supported for fellowship by the Women Scientist Scheme A (WOS-A) of the Department of Science and Technology, India,
through Project No. SR/WOS-A/PM-18/2016-2019.

%\section*{Funding}
%No funding available.
%
%\section*{Availability of data and materials}
%Not applicable.
%
%\section*{Competing interests}
%The authors declare that they have no competing interests.
%
%\section*{Authors' contribution}
%Both the authors contributed equally.

\end{document}